\documentclass[12pt]{amsart}
\usepackage[left=2.2cm,tmargin=2.5cm,bmargin=2.5cm,right=2.2cm]{geometry}

\usepackage{amsmath}% http://ctan.org/pkg/amsmath
\DeclareRobustCommand{\stirling}{\genfrac\{\}{0pt}{}}

\usepackage{hyperref}

\usepackage{lmodern}
\usepackage[T1]{fontenc}

\usepackage{textcomp}

\newtheorem{theorem}{Theorem}

\newtheorem{lemma}[theorem]{Lemma}
\newtheorem{cor}[theorem]{Corollary}

\theoremstyle{remark}

\author{Will Sawin}
\title[Lower bounds for multicolor Ramsey numbers]{An improved lower bound for multicolor Ramsey numbers and a problem of Erd\H{o}s}

\begin{document}

\maketitle

\begin{abstract} The multicolor Ramsey number problem asks, for each pair of natural numbers $\ell$ and $t$, for the largest $\ell$-coloring of a complete graph with no monochromatic clique of size $t$. Recent works of Conlon-Ferber and Wigderson have improved the longstanding lower bound for this problem. We make a further improvement by replacing an explicit graph appearing in their constructions by a random graph. Graphs useful for this construction are exactly those relevant for a problem of Erd\H{o}s on graphs with no large cliques and few large independent sets. We also make some basic observations about this problem.\end{abstract}

Let $r(t;\ell)$ be the multicolor Ramsey number, that is, the minimum $r$ such that every $\ell$-coloring of the edges of the complete graph $K_r$ on $r$ vertices has a monochromatic clique of size $t$.

Building on a breakthrough of Conlon and Ferber \cite{ConlonFerber}, Wigderson \cite{Wigderson} proved the lower bound
\[ r(t; \ell) \geq  2^{  \frac{ 3 (\ell-2)  t}{8} +  \frac{t }{2} + o(t)  }   .\]

Wigderson's variant of Conlon and Ferber's construction involves an explicit graph constructed from linear algebra over $\mathbb F_2$, a series of $\ell-2$ random maps from the vertex set of $K_r$ to the vertex set of this graph, and a random coloring with the last two colors of those edges of $K_r$ not sent to edges of this graph by any of these maps. 

We give the slight improvement
\[ r(t,\ell) \geq 2^{ .383796 (\ell-2) t+  \frac{t}{2} + o(t)} ,\] with the exponent $.383796$ improving over $\frac{3}{8} =.375$. We obtain this by replacing the explicit linear algebraic graph with an Erd\H{o}s-R\'enyi random graph, with the probability that each edge lies in the graph slightly less than $1/2$. (If we took the probability to be exactly $1/2$, we would recover Wigderson's bound \cite{Wigderson}). Thus our argument follows, in part, the suggestion of Alon and R\"odl \cite[p. 140]{AlonRodl} to combine random homomorphisms with random graphs for the multicolor Ramsey number problem.

We further analyze this approach to multicolor Ramsey numbers by relating it to a 1962 problem of Erd\H{o}s, which asks for the asymptotics in $n$ of the minimum number $f(n,s,t)$ of independent sets of size $s$ in a graph $G$ with $n$ vertices that contains no clique of size $t$.  This problem is approximately equidistant between the classical two-color Ramsey problem, which asks for the largest graph with no independent set of size $t$ and no clique of size $s$, and the Ramsey multiplicity problem (\cite{ErdosComplete}, see also \cite[\S2.6]{ConlonFoxSudakov}), which asks us to minimize the total number of independent sets of size $t$ and cliques of size $t$. 

The following is an equivalent formulation: For natural numbers $s$ and $t$, define  $c_{s,t} $ to be the infimum over graphs $G$ with no clique of size $t$ of the probability that, for $v_1,\dots, v_s$ vertices of $G$ chosen independently and uniformly at random, $\{v_1,\dots, v_s \}$ is an independent set.

By the blow-up construction first applied to this problem by Nikiforov \cite[\S2]{Nikiforov}, it is straightforward to check that $c_{s,t} = \lim_{n \to \infty} f(n,s,t)/ \binom{n}{s}$, but this formulation will be more convenient for the application to multicolored Ramsey problems.

%The advantage of choosing $v_1,\dots, v_s$ independently, which incorporates independent sets of size $<s$ in the definition of $c_{s,t}$, is that it avoids the need to consider a limit as the graph size $n$ goes to $\infty$.  This variant is also particularly compatible with the random homomorphism approach to multicolored Ramsey problems.

We prove a small number of basic results about $c_{s,t}$ that are close analogues of known results for the usual Ramsey and Ramsey multiplicity problems. In particular, we give a lower bound for $c_{s,t}$ that appears to be new for general $s,t$ (though does not improve on the exact values that are known for many particular values of $s,t$).

In addition, we observe, as in the Ramsey multiplicity problem, but unlike the usual Ramsey problem, that explicit graphs  provide equally good values as random graphs for certain special values of $s/t$. For example, the graph studied by Conlon and Ferber is as efficient as a random graph for $s =3t/4$. It would be interesting to see if this is true for other ratios of $s$ and $t$.

This article was written while the author served as a Clay Research Fellow and finished while the author was supported by NSF grant DMS-2101491. I would like to thank Yuval Wigderson, D\"om\"ot\"or P\'alv\"olgyi, Benny Sudakov, and the anonymous referee for helpful comments. 

\section{Lower bounds and application to multicolor Ramsey numbers}

Our first lemma relates $c_{t,t}$ to the multicolor Ramsey numbers. It is a slight variant of \cite[Lemma 3.1]{WigdersonNotes}.

\begin{lemma}\label{to-multicolor}  For any $\ell, t \geq 2$, we have \[  r(t; \ell) \geq c_{t,t}^{ - \frac{\ell-2}{t}} 2^{  \frac{t-1}{2}} .\] \end{lemma}
 
\begin{proof} Let $N =\lfloor c_{t,t}^{ - \frac{\ell-2}{t}} 2^{  \frac{t-1}{2}} \rfloor $. It suffices to find an $\ell$-coloring of $K_N$ with no monochromatic clique of size $t$. 

Fix a graph $G$ with no clique of size $t$ and such that the probability that $t$ random vertices of $G$ form an independent set is at most $c_{t,t}+\epsilon$ for $\epsilon>0$ to be chosen later depending on $\ell, t$.

We construct such an $\ell$-coloring of $K_N$ as follows: We choose $\ell-2$ functions $f_1,\dots, f_{\ell-2}$ from the vertex set of $K_N$ to the vertex set of $G$ at random, independently and uniformly. We color the edge between a pair of vertices $x,y$ the $i$th color if $i$ is minimal such that $f_i(x)$ and $f_i(y)$ are connected by an edge of $G$. If there is no $i$ such that $f_i(x)$ and $f_i(y)$ are connected by an edge of $G$, then we randomly color this pair of vertices with either the $\ell-1$st color or $\ell$th color, each with probability $\frac{1}{2}$, independently for each remaining pair.

We claim this coloring has no monochromatic cliques with positive probability.

Since $G$ has no clique of size $t$, there is no set of $t$ vertices sent to a clique of size $t$ by $f_i$, and thus the $i$th color has no cliques of size $t$. So it suffices to show that the probability that there is a clique of size $t$ of the last two colors is less than $1$. Fix a set $S$ of $t$ vertices of $K_N$. For $S$ to be a clique of the $\ell-1$st or $\ell$th color, it must first have no edges of the first $\ell-2$ colors, so we must not have $f_i(x)$ and $f_i(y)$ connected by an edge of $G$ for any $i$. By assumption, this occurs with probability at most $(c_{t,t} +\epsilon)^{\ell-2}$.  Furthermore, the randomly selected colors for the edges must be equal, which occurs with probability $2^{ 1 - t (t-1)/2}$. 

By the union bound, the probability that there is any monochromatic clique is thus at most
\[  \binom{N}{t} (c_{t,t} +\epsilon)^{\ell-2}2^{ 1 - t (t-1)/2} <  \frac{N^t}{t!} c_{t,t}^{\ell-2}2^{ 1 - t (t-1)/2}  \leq N^tc_{t,t} ^{\ell-2}2^{  - t (t-1)/2} \leq 1 \] where we choose $\epsilon$ sufficiently small to make the first inequality true, which is always possible as $\binom{N}{t} < \frac{N^t}{t!}$, and the last inequality is by the definition of $N$. 

\end{proof}

%Combining Lemma \ref{f2-construction} and Lemma \ref{to-multicolor}, we obtain  \[ r(t;\ell) = O (1)^{ \frac{\ell-2}{t}}  2^{  \frac{ 3 (\ell-2)(t-2)}{8} + \frac{ t-1}{2}} = O (1)^{ \frac{\ell-2}{t}}  2^{  \frac{ 3 \ell (t-2)}{8} - \frac{t-4 }{4} } ,\] recovering the bound of Wigderson in slightly more explicit form.

%However, we can do better using a probabilistic construction.

Our next lemma gives a probabilistic upper bound for $c_{s,t}$. Here, and in this paper, all logarithms are to base $2$.

\begin{lemma}\label{random-construction} For any $p \in [0,1]$, for any $s,t$, we have \begin{equation}\label{eq-random-construction} c_{s,t} \leq    2^{ \frac{ t (4 s \log(1-p) - t \log(p) )\log(p)  }{8 \log(1-p) }  + s \log s + O(t) }.\end{equation} \end{lemma}

\begin{proof}  Let $G$ be a random graph on $M$ vertices, with $M$ to be chosen later, where each pair of vertices is connected by an edge with probability $p$. Let $v_1,\dots, v_s$ be uniformly distributed random variables in $[M]$, independent from each other and from $G$. We have
\[ c_{s,t} \leq \frac{ \mathbb P ( \{v_1,\dots, v_s\} \textrm{ is an independent set of }G ) } { \mathbb P ( G \textrm{ contains no clique of size }t)}\]
since this is an upper bound for the expectation over $G$ of the probability that $\{v_1,\dots, v_s\}$ is independent conditional on $G$ containing no clique of size $t$, and we can always choose a graph that attains at most the conditional expectation.

We estimate the probabilities in both the numerator and the denominator using the union bound. For the denominator, the probability that $G$ does contain a clique of size $t$ is at most $ \binom{M}{t} p^{ \frac{t(t-1)}{2}} \leq \frac{M^t}{t!} p^{ \frac{t(t-1)}{2}} = o(1) $ if we take $M =  \lceil p^{-t/2} \rceil$. Since $\frac{1}{ 1- o(1)} = O(1)$, the denominator contributes an $O(1)$ factor.

For the numerator, taking $k$ to be the cardinality of $\{v_1,\dots, v_s\}$, we have
\[ \mathbb P ( \{v_1,\dots, v_s\} \textrm{ is an independent set of }G  )  \] \[\leq \sum_{k=0}^s  \frac{ \binom{M}{k} \stirling{s}{k} k!}{ M^s}  (1-p)^{ \frac{k (k-1)}{2} } \leq \sum_{k=0}^s  \frac{ \stirling{s}{k} }{ M^{s-k} }  (1-p)^{ \frac{k (k-1)}{2}} \] \[ \leq s! \max_{0 \leq k \leq s} \frac{ (1-p)^{ \frac{k (k-1)}{2}} } { p^{-(s-k) t/2 }}  \leq s! \max_{0 \leq k \leq s} 2^{ k(k-1) \log (1-p)/2 + (s-k) t \log (p)/2}  \]
 where $\stirling{s}{k} $ are the Stirling numbers of the second kind, since $\sum_{k=0}^s \stirling{s}{k} \leq s!$.
 
 Now the quadratic $k(k-1) \log (1-p)/2 + (s-k) t \log (p)/2$ is maximized by taking $k =  \frac{ t \log p}{ 2\log (1-p)} + \frac{1}{2} $   and attains a maximum value of 
 %\[  \left( \left( \frac{ t \log p}{ 2\log (1-p)} \right)^2-  \frac{1}{4}\right) \frac{\log(1-p)}{2}  + \left( s - \frac{ t\log p}{ 2\log (1-p)}  - \frac{1}{2} \right) \frac{t}{2} \log p\]
  \[   \left( \frac{ t \log p}{2 \log (1-p)} \right)^2\frac{\log(1-p)}{2}  + \left( s - \frac{ t \log p}{ 2\log (1-p)}  \right ) \frac{t}{2} \log p  + O(t) \] \[=  \left( s - \frac{ t \log p}{ 4\log (1-p)}  \right ) \frac{t}{2} \log p  + O(t)   = \frac{ t (4 s \log(1-p) - t \log(p) )\log(p)  }{8 \log(1-p) }  +O (t) .\]
 The $s!$ term is bounded by $2^{ s \log s}$ and the $O(1)$ factor can be absorbed into $2^{ O(t)}$. This gives \eqref{eq-random-construction}.
 \end{proof}
 
 In this proof, when $ \frac{ t \log p}{ 2\log (1-p)} + \frac{1}{2} >s$, we could get a better bound using the constraint $k \leq s$, but this is never needed as increasing $p$ would always give an even better bound in that range.

% Specializing to $s=t$, Lemma \ref{random-construction} gives
%
%\[ c_{t,t} \leq 2^{ t^2 \theta + o(t^2)} \] where 
%
%\[\theta = \min_{p \in [0,1]}  \frac{ (4 \log(1-p) - \log(p) ) \log(p)}{ 8 \log(1-p) } \approx -.383796\]
%attained for $p=.454997$.

 \begin{cor} We have \[ r(t,\ell) \geq 2^{ .383796 (\ell-2) t+  \frac{t}{2} + o(t)} .\] \end{cor}
 
 \begin{proof} Combining Lemma \ref{to-multicolor} and Lemma \ref{random-construction}, we have
 \[  r(t; \ell) \geq c_{t,t}^{ - \frac{\ell-2}{t}} 2^{  \frac{t-1}{2}}  \geq 2^{ -\frac{ (\ell-2) t (4  \log(1-p) -  \log(p) )\log(p)  }{8 \log(1-p) }  + (\ell-2)  \log t+ O(\ell -2 ) +  \frac{t-1}{2}}  .\]
 Taking $p=.454997$, the constant $ - \frac{(4  \log(1-p) -  \log(p) )\log(p)  }{8 \log(1-p)} $ becomes $.383796$, and the other terms, except for $2^{ \frac{t}{2}}$, are lower-order and may be absorbed into $2^{ o(t)}$.  (Here $p=.454997$ is obtained by maximizing $ -\frac{(4  \log(1-p) -  \log(p) )\log(p)  }{8 \log(1-p)}$.)
 \end{proof}

\section{Additional estimates on $c_{s,t}$} 

There is a close relationship between the probabilities $c_{s,t}$ and the (two-colored) Ramsey numbers $R(s,t)$, a variant of the relationship between the usual Ramsey and Ramsey multiplicity problems established in \cite[Equation (4)]{ErdosComplete}, and also a variant of the relationship with $R(3,t)$ given in \cite[Theorem 2.5]{Nikiforov}.

\begin{lemma}\label{to-ramsey} For any natural numbers $s,t$, we have
\[ c_{s,t} \geq \frac{1}{ \binom{ r(s,t) }{ s}}\]
and for any $a<s$, we have  
\[ c_{s,t} \leq   \binom {r(a,t) -1 }{a-1}   \left( \frac{a-1}{ r(a,t)-1} \right)^s  .\] \end{lemma}

\begin{proof}
For the first inequality, there exists by definition a graph $G$ with no cliques of size $t$ and where the probability that $s$ random elements form a clique is at most $c_{s,t}+\epsilon$. Consider a random map from the vertex set of the complete graph of size $r(s,t)$ to $G$. The inverse image of the edge set of $G$ under this map has no cliques of size $t$, and its expected number of cliques of size $s$ is $\binom{ r(s,t) }{ s} (c_{s,t}+\epsilon)$, which must be $\geq 1$ since, by the definition of $r(s,t)$, all such graphs have a clique of size $s$. Taking $\epsilon$ arbitrary small, we obtain the first inequality.

For the second inequality, fix a graph $G$ with $r(a,t)-1$ vertices containing no cliques of size $t$ and no independent sets of size $a$. In this graph, $s$ elements form an independent set only if they are contained in a set of vertices of size $a-1$, so the probability that $s$ random elements form an independent set is at most the number $\binom { r(a,t) -1 }{a-1}  $ of sets of size $a-1$ times the probability $ \left( \frac{a-1}{ r(a,t)-1} \right)^s $ that they are all contained in one such set.
\end{proof} 

We do not expect the bounds from Lemma \ref{to-ramsey} to be useful in many concrete cases. Instead, we should get better bounds for $c_{s,t}$ by taking ideas from Ramsey theory and applying them to $c_{s,t}$, but not so much better that they imply stronger bounds for the original Ramsey number problems.

Motivated by this, we prove a lower bound for $c_{s,t}$ (analogous to the upper bound for $r(s,t)$ proved using the neighborhood method \cite{ErdosSzekeres}):

\begin{lemma}\label{neighborhood-induction} For any natural numbers $t,s$, we have
\[ c_{s,t} \geq  \min_{ x \in [0,1]} ( \max ( x^s c_{s,t-1} , (1-x)^{s-1} c_{s-1,t} )),\]
\end{lemma}

Compare the inequality \[ r (s,t) \leq r ( s, t-1) + r(s-1,t) \] for Ramsey numbers, which can be equivalently expressed as  \[ r(s,t) \leq \max_{ x \in [0,1] } ( \min ( x^{-1} r(s, t-1), (1-x)^{-1} r(s-1, t) )).\]

\begin{proof} It suffices to show that for each graph $G$ lacking cliques of size $t$, the probability that $s$ uniformly random vertices form an independent set is at least $ \min_{ x \in [0,1]} ( \max ( x^s c_{s,t-1} , (1-x)^{s-1} c_{s-1,t} ))$.

Fix such a graph $G$. Let $x= \Delta(G) / |G|$ where $\Delta(G)$ is the maximum degree of a vertex in $G$. 

If we let $v$ be a vertex of maximal degree and $G_1$ the subgraph of $G$ consisting of vertices connected to $v$ by an edge, then $G_1$ has no clique of size $t-1$, so the probability that $s$ vertices sampled uniformly at random from $G_1$ form a clique is at least $ c_{s, t-1}$. Since $s$ random elements of $G$ have a $x^s$ probability of all lying in $G_1$, and conditionally on lying in $G_1$, are uniformly distributed in $G_1$, it follows that the probability that $s$ random elements of $G$ form an independent set is at least $x^s c_{s,t-1}$.

Let us next show that the  probability that $s$ random elements $v_1,\dots,v_s$ of $G$ form an independent set is at least $(1-x)^{s-1} c_{s-1,t}$.  In fact we will show this for $v_1$ arbitrary and $v_2,\dots, v_s$ uniformly random. Let $G_2$ be the subgraph of $G$ consisting of vertices not connected to $v_1$ by an edge, including $v_1$. By the definition of $x$, $|G_2| \geq (1-x)|G|$ and so the probability that $v_2,\dots, v_s$ lie in $G_2$ is at least $(1-x)^{s-1}$. Conditionally on this, they are uniformly distributed in $G_2$, which has no clique of size $t$, so the conditional probability that they form an independent set is at least $c_{s-1,t}$, and because none has an edge to $v_1$, they form an independent set with $v_1$ as well, with probability at least $(1-x)^{s-1}c_{s-1,t}$.

Thus, for some $x$, the probability that $s$ vertices in $G$ form an independent set is at least $ \max ( x^s c_{s,t-1} , (1-x)^{s-1} c_{s-1,t} )$, and it follows that the probability is at least $ \min_{ x \in [0,1]} ( \max ( x^s c_{s,t-1} , (1-x)^{s-1} c_{s-1,t} ))$, as desired. \end{proof}

Let $N(s,t)$ be the unique function that satisfies $N(s,t)=1$ if $s=1$ or $t=2$ and \[ N(s,t) =  \min_{ x \in [0,1]} ( \max ( x^s N(s,t-1) , (1-x)^{s-1} N(s-1,t) ))\] if $s>1$ and $t>2$. 

\begin{lemma} We have \begin{equation}\label{neighborhood-formal} c_{s,t} \geq N(s,t).\end{equation} \end{lemma}

Thus $N(s,t)$ describes a lower bound for $c_{s,t}$ obtained by the neighborhood method. This is the analogue for $c_{s,t}$ of \cite[Theorem 4]{ConlonMultiplicity} in the Ramsey multiplicity problem.

\begin{proof} This follows by induction. For the base cases $s=1$ and $t=2$, we observe $c_{1,t}=1$ because one vertex always forms an independent set and $c_{s,2}=1$ because a graph with no cliques of size $2$ has no edges and thus every set of vertices is independent.

For the induction step, using Lemma \ref{neighborhood-induction} and the induction hypothesis, we have
\[ c_{s,t} \geq  \min_{ x \in [0,1]} ( \max ( x^s c_{s,t-1} , (1-x)^{s-1} c_{s-1,t} ))\] \[\geq   \min_{ x \in [0,1]} ( \max ( x^s  N(s,t-1) , (1-x)^{s-1} N(s-1,t) )) = N(s,t). \qedhere \] \end{proof}

It is not clear if there exists a closed-form formula for $N(s,t)$, even approximately. Instead, we have the following lower bound, which shows that for $t$ a constant multiple of $s$, $N(s,t)$ is exponential in $-s^2$.

\begin{lemma}\label{neighborhood-binomial} We have
\[ N(s,t) \geq  \left( \frac{ \left(  \frac{s-1}{2} +t-2\right)^{\frac{s-1}{2} +t-2 } }{ \left( \frac{s-1}{2} \right)^{\frac{s-1}{2} } (t-2)^{t-2} } \right)^{-s}.\]\end{lemma} % \left( \frac{ (s/2+t)^{s/2 +t} }{ (s/2)^{s/2} t^t} \right)^{-s} .\] \end{lemma}

In particular, $N(t,t) \geq  \left( \frac{3 \sqrt{3}}{2} \right)^{-t^2} = 2^{ -1.37744 \dots t^2}$, so our upper and lower bounds for $N(t,t)$ differ by a factor of slightly less than $4$ in the exponent. This is the analogue for $c_{s,t}$ of \cite[Theorem 2]{ConlonMultiplicity}. It is likely possible to instead prove in this setting an analogue of the sharper bound \cite[Theorem 1]{ConlonMultiplicity}, by a similar analytic method, but we don't pursue this here. 

\begin{proof} We have
\[ N(s,t) =  \min_{ x \in [0,1]} ( \max ( x^s N(s,t-1) , (1-x)^{s-1} N(s-1,t) )) \] \[= \max_{ x \in [0,1]} ( \min ( x^s N(s,t-1) , (1-x)^{s-1} N(s-1,t) ))\]
since $ x^s N(s,t-1) $ is monotonically increasing in $x$ and $ (1-x)^{s-1} N(s-1,t) ))$ is monotonically decreasing, so both the min-max and the max-min are attained at the unique point where they are equal. Thus, for any $y\in [0,1]$
\[ N(s,t) \geq \min ( y^s N(s,t-1) , (1-y)^{s-1} N(s-1,t) ).\]

We can now prove that, for any $y\in [0,1]$, \[ N(s,t) \geq (1-y)^{ \frac{s (s-1)}{2}} y^{ s (t-2)}\] by induction on $s,t$, with base cases $s=1$ because $ 1 \geq y^{t-2}$ and $t=2$ because $1 \geq (1-y)^{ \frac{s(s-1)}{2}}$ and induction step
\begin{align*} N(s,t) \geq  \min ( y^s N(s,t-1) , &(1-y)^{s-1} N(s-1,t) ) \\ \geq \min ( y^s (1-y)^{  \frac{s(s-1)}{2} } y^{ s(t-3)} ,& (1-y)^{s-1} (1-y)^{ \frac{(s-1) (s-2)}{2}} y^{ (s-1) (t-2)}) \\= \min(   (1-y)^{ \frac{s (s-1)}{2}} y^{ s (t-2)}, &(1-y)^{ \frac{s (s-1)}{2}} y^{ (s-1)  (t-2)}) =  (1-y)^{ \frac{s (s-1)}{2}} y^{ s (t-2)}.\end{align*}

%By induction, $N(s,t)$ is at least the minimum over sequences of pairs $s_i, t_i$ where $s_0 =1$ or $t_0=2$,  $(s_{i+1},t_{i+1}) = (s_i+1, t_i)$ or $(s_i, t_i+1)$, and $(s_n,t_n)=(s,t)$ of $\prod_{ \{ i \mid s_{i+1}=s_{i} +1\}  } (1-y)^{ s_i}  \cdot \prod_{\{ i \mid t_{i+1} = t_i +1 \}} y^{s_i} $. This minimum is attained for the sequence $(1,2), (2,2), (3,2),\dots, (s,2), (s,3),\dots, (s,t)$, where it is
%\[ (1-y)^{ \frac{s (s-1)}{2}} y^{ s (t-2)} \]

Taking $y =  \frac{ t-2}{   \frac{s-1}{2} + t-2 } $, we obtain 
\[ N(s,t) \geq  \left( \frac{ \left(  \frac{s-1}{2} +t-2\right)^{\frac{s-1}{2} +t-2 } }{ \left( \frac{s-1}{2} \right)^{\frac{s-1}{2} } (t-2)^{t-2} } \right)^{-s} .\]

\end{proof} 

The upper bound for $c_{s,t}$ arising from the probabilistic method is Lemma \ref{random-construction}. In addition to the probabilistic method, we have access to an explicit construction, first applied to this problem by Conlon and Ferber \cite{ConlonFerber}. The bound obtained this way is as follows.

\begin{lemma}\label{f2-construction} For $t$ even, we have 

\[ c_{s,t} \leq \begin{cases}  O \left( 2^{ - \frac{s (s-1) }{ 2} }\right) & s \leq \frac{t}{2}-1 \\ O \left( 2^{ -\frac{ (4s-t)(t-2)}{8}} \right) & s \geq \frac{t}{2}-1 \end{cases}\]

\end{lemma}

This is a reformulation of estimates from \cite{ConlonFerber} and \cite{Wigderson}.

A weaker form of the small $s$ case follows by results of \cite{ChungGrahamWilson} from the fact that the graph $G$ described below is quasi-random, which is known by \cite[Lemma 4.1]{PudlakRodlSavicky} or \cite{KrivelevichSudakov}.

\begin{proof} Let $V$ be a vector space of dimension $t-2$ over $\mathbb F_2$ endowed with a nondegenerate symplectic bilinear form $(\cdot,\cdot)\colon V \times V \to \mathbb F_2$. Let $G$ be the graph whose vertex set is $V$ and with an edge between $v_1$ and $v_2$ if and only if $(v_1,v_2)=1$.

Then $G$ contains no clique of size $t$. To prove this, note that $t$ vectors in a clique would have to satisfy $(v_i, v_j) =  1- \delta_{ij}$. Because $t$ is even, the matrix with entries $1-\delta_{ij}$ has rank $t$, and thus $v_1,\dots,v_t$ would have to be linearly independent, contradicting the fact that $V$ does not contain $t$ linearly independent vectors.

Any independent set of $G$ is contained in an isotropic subspace of $V$, since if the form vanishes on a set of vectors then it vanishes on all linear combinations. Since every isotropic subspace of $V$ of dimension $<t/2$ is contained in a subspace of dimension one greater, and every isotropic subspace of dimension $t/2-1$ is maximal, any independent set of size $s$ is contained in an isotropic subspace of dimension $\min (s, t/2-1)$.

If $s \leq  \frac{t}{2} -1 $, note that the number of isotropic subspaces in $V$ of rank $s$ is  \[ \frac{\prod_{i=0}^{s-1} (2^{t-2 -2i } - 1) }{ \prod_{i=0}^{s-1} (2^{s-i} - 1  ) }  \leq O\left( 2^{   \frac{s (2t-3-3s )}{2}} \right ) \]  since in a basis for a maximal isotropic subspace we have $2^{t-2}-1$ choices for the first vector, $2^{t-3}- 2$ choices for the second, $2^{ t-4}-4$ choices for the third, and so on, and the number of isotropic subspace is the number of such bases divided by the number $(2^s-1)  (2^s-2) \dots (2^s - 2^{s-1})$ of bases of a single isotropic subspace.

The probability that $s$ vectors will be contained in a given isotropic subspace of rank $s$ is $2^{ - (t-2-s) s}$, so the total probability that $s$ vectors will form an independent set is $ O\left( 2^{ -  \frac{s (s-1) }{2}} \right) $ if $s\leq \frac{t}{2}-1$.

In particular, the number of $(\frac{t}{2}-1)$-dimensional isotropic subspaces is $O\left( 2^{   \frac{ t (t-2) }{8}} \right)$, and the probability that $s$ vectors will be contained in one is $ 2^{ - (t-2) s/2}$, so the total probability that $s$ vectors will form an independent set is $O \left( 2^{ - (4s-t) (t-2) / 8 } \right)$ if $s \geq \frac{t}{2}-1$.  \end{proof} 

To compare to earlier work, note that the bilinear form $v_i \cdot v_j$ on the space of vectors in $\mathbb F_2^t $ with Hamming weight even is symplectic, but not nondegenerate, since the all-ones vector has zero dot product with any vector. Taking the quotient by this vector produces a vector space of rank $t-2$ with a nondegenerate symplectic bilinear form.

Specializing the random construction of Lemma \ref{random-construction} to $p=1/2$, we obtain \[c_{s,t} \leq 2^{ - \frac{ t (4s-t) }{ 8} + s \log s + O(t)} = 2^{ - \frac{t (4s-t)}{8} + O (s\log(s)+t)} ,\] which matches up to lower-order terms the bound obtained from Lemma \ref{f2-construction}.

In Lemma \ref{random-construction}, the value of $p$ is optimal if it minimizes $\frac{ t (4 s \log(1-p) - t \log(p) )\log(p)  }{8 \log(1-p) } $, which happens if \[  -\frac{4s}{1-p} - \frac{t}{p} + (4s \log(1-p) - t\log(p)) \left( \frac{1}{ p \log(p) } + \frac{ 1}{ (1-p) \log(1-p)}  \right)= 0 .\] For $p=1/2$, this specializes to \[ - 4s -t + 2 (4s-t) = 0\] or $s =3t/4$. 

Thus, when $s = 3t /4 $, the explicit Lemma \ref{f2-construction} gives the same bound as the random Lemma \ref{random-construction} up to lower-order terms. It would be interesting to determine if explicit constructions can match Lemma \ref{random-construction} for other values of $s/t$. It is possible that the graph, also studied by Conlon and Ferber \cite{ConlonFerber}, with vertex set the set of vectors  in $v\in \mathbb F_q^n$ with $v\cdot v=0$ and an edge connecting two vertices $v,w$ if $v\cdot w=1$, for a finite field $\mathbb F_q$, would give the same bound as Lemma \ref{random-construction} for $p=1/q$.

The exact optimum values of $c_{s,t}$ are known for many small $s,t$. In most known cases, the optimal value is attained by a graph which also optimizes some Ramsey number problem. We review some known results here: 

\begin{itemize}

\item We have $c_{2,t} =\frac{1}{t-1}$, with the optimum attained by the complete graph $K_{t-1}$. This follows immediately from Tur\'an's theorem. 
\item We have $c_{3,3} = \frac{1}{4}$, with the optimum attained by the complete graph $K_2$ \cite{Goodman}. 
\item We have $c_{3,t}=\frac{1}{ (t-1)^2}$, with the optimum attained by the complete graph $K_{t-1}$, for $4 \leq t \leq 7$ \cite{PikhurkoVaughan} (with the case $t=4$ discovered independently by \cite{DHMNS}). However, this is known not to hold for $t$ large enough \cite[Corollary 2.6]{Nikiforov}.
\item We have $c_{4,3}= \frac{3}{25}$, with the optimum attained by the five-cycle (independently due to \cite{PikhurkoVaughan} and \cite{DHMNS}).
\item We have $c_{5,3}=31/625$, with the optimum again attained by the five-cycle \cite{PikhurkoVaughan}.
\item We have $c_{6,3}=19211/2^{20}$ and $c_{7,3}=98491/2^{24}$, with the optimum attained by the (5-regular) Clebsch graph \cite{PikhurkoVaughan}.
\end{itemize}

Except for the first two, these results were obtained using the method of flag algebras introduced by Razborov \cite{Razborov}. If these methods could be extended to more values of $s,t$, it would be interesting to see to what extent the optimal graphs have algebraic structure, which could suggest non-random constructions for large $s,t$ which generalize Lemma \ref{f2-construction}.

\bibliographystyle{plain}

\bibliography{references}

\end{document}